\documentclass[]{article}

\usepackage{cancel}

\usepackage{array}
\usepackage{amsmath}
\usepackage{amsthm}
\usepackage{amssymb}
\usepackage{graphicx}
\usepackage{MnSymbol}
\usepackage{amsfonts}
\usepackage{tabularx}
\usepackage{yfonts}
\usepackage{cancel}
\usepackage{makeidx}
\usepackage{url}
\usepackage{fancyhdr}
\usepackage{enumerate}

\usepackage[noblocks]{authblk}

\newenvironment{mainthrm*}
{\noindent {\bf Main Theorem.}
}
{
}

\theoremstyle{definition} \newtheorem{kap}{Definition}

\newtheorem{prop}{Proposition}

\def\+dot
{
\stackrel{.}{+}
}

\title{Irreducible Truth-Value Algebras Suffice for the \\ Completeness of Many First-Order Algebraic Logics}

\author{Richard DeJonghe, Kimberly Frey, and Tom Imbo$^{*}$ \\ {\small \emph{Department of Physics}} \\ {\small \emph{University of Illinois at Chicago}} \\ {\small \emph{845 W.~Taylor St., Chicago, IL 60607}}}

\date{}

\begin{document}

\maketitle

\begin{abstract}
It is well-known that a Hilbert-style deduction system for first-order classical logic is sound and complete for a model theory built using all Boolean algebras as truth-value algebras if and only if it is sound and complete for a model theory utilizing only irreducible Boolean algebras (which are all isomorphic to the two-element Boolean algebra $\mathbf{2}$).  In this paper, we prove an analogous result for \emph{any} first-order logic with an algebraic semantics satisfying certain minimal assumptions, and we then apply our result to first-order quantum logic.

\end{abstract}

\section{Introduction} \
\label{intro}

While first-order classical logic is typically developed using the two element Boolean algebra $\mathbf{2}$ as its algebra of truth values,  one obtains an equivalent logic by utilizing \emph{all} Boolean algebras.  More specifically, a formal deduction system is sound and complete for a model theory built using $\mathbf{2}$ as the truth value algebra if and only if it is sound and complete for a model theory utilizing all Boolean algebras \cite{Manin}.  A Boolean algebra is \emph{irreducible} if it is not isomorphic to a product of two Boolean algebras, and $\mathbf{2}$ is the \emph{only} non-trivial irreducible Boolean algebra (up to isomorphism).  Hence, we can re-state the above theorem as follows --- a formal deduction system is sound and complete for a model theory utilizing all Boolean algebras if and only if it is sound and complete for a model theory utilizing only irreducible Boolean algebras.  

There are many first-order logics other than classical which are naturally associated with a certain class of algebras --- some of these utilize the same syntactical structure as classical first-order logic but change the axioms (e.g.~quantum logic, intuitionistic logic, etc.), while others modify the relevant set of logical connectives (e.g.~modal logic, conditional logic, etc.).  It is natural to wonder whether an analogous result (to that for classical logic described above) concerning irreducible truth-value algebras applies to any such non-classical logic.  In this paper, we show that a first-order logic is sound and complete with respect to a semantics based on a class of algebras $\mathcal{K}$ if and only if it is sound and complete with respect to a semantics based on the irreducible algebras in $\mathcal{K}$.  The only assumptions in our proof are that the universal quantifier is evaluated in the semantics via a partial order on the algebras in $\mathcal{K}$, and that the class of models defining the semantics is reasonably well-behaved. (See below for a precise statement of our result.)

The plan of this paper is as follows.  In section \ref{sec:prelim} we introduce some preliminaries concerning truth-value algebras and define our model theory.  Following this, we prove our main result in section \ref{sec:main}.  In section \ref{QL} we apply this result to quantum logic.  Finally, in section \ref{sec:conc} we note some open questions.

\noindent\rule{6cm}{0.4pt}

{\footnotesize * imbo@uic.edu }

\section{Preliminary Material}
\label{sec:prelim}

\subsection{Truth Values}

We now discuss the truth values in the logics we consider, which will have both an order-theoretic as well as an algebraic structure.

\subsubsection{Posets}

Recall that a poset $(P,\leq)$ is a set $P$ along with a binary relation $\leq$ on $P$ that is reflexive, anti-symmetric, and transitive.   For a given poset $(P, \leq)$, and for a given $A \subseteq P$, we will use $\bigvee A$ and $\bigwedge A$ to represent the least upper bound and greatest lower bound for $A$, respectively (when such an element exists), or alternatively we may use the notation $\bigvee_{a \in A} a$ and $\bigwedge_{a \in A} a$, depending on context.  For a poset $P$, $\bigvee P$ is called the \emph{top element} (when it exists), and is denoted by $1$ (when $P$ is clear from the context).  

Given two posets $(P_1, \leq_1)$ and $(P_2, \leq_2)$, we call a map $g:P_1 \to P_2$ \emph{isotone} if $g(a) \leq_2 g(b)$ for all $a,b \in P_1$ such that $a \leq_1 b$.  An isotone map $g:P_1 \to P_2$ will further be called \emph{continuous} if, for any $A \subseteq P_1$ for which $\bigwedge A$ exists, we have that $\bigwedge g(A)$ exists,\footnote{For any map $g:P_1 \to P_2$, we define the notation $g(A) := \{ g(a) \ : \ a \in P_1 \}$. } and moreover that $\bigwedge g(A) = g( \bigwedge A)$.  Finally, if the two posets above have top elements $1_1$ and $1_2$, respectively, the map $g$ is called \emph{top-preserving} if $g(1_1) = 1_2$.

\subsubsection{Universal Algebra}

First, let $\mathbb{N} = \{0,1,2,\ldots\}$.  A \emph{type} is a pair $(\alpha,\tau)$, where either $\alpha = \{0,1,\ldots,n\}$ for some $n \in \mathbb{N}$, or $\alpha = \mathbb{N}$, and where $\tau : \alpha \to \mathbb{N}$ is such that $i \leq j$ implies $\tau(i) \geq \tau(j)$ for all $i,j \in \alpha$.   An \emph{algebra of type $(\alpha,\tau)$} is then a tuple $(A,O,\ell)$, where $A$ is a set, $O$ is a finite or countably infinite set (called the \emph{operations} of the algebra), $\ell:O \to \alpha$ is a bijective map (called the \emph{labeling function}), and $f:A^{\tau(\ell(f))} \to A$ for each $f \in O$.  

For notational convenience, when both the algebra $(A,O,\ell)$ and its type $(\alpha, \tau)$ are clear from the context, for any $f \in O$ we call $N_f :=  \tau(\ell(f))$ the \emph{arity of $f$}.

Note that for any two algebras $(A_1,O_1,\ell_1)$ and $(A_2,O_2,\ell_2)$  of the same type $(\alpha,\tau)$, $O_1$ and $O_2$ are naturally isomorphic (as sets) via the map\footnote{The symbol ``$\circ$'' denotes composition of maps.} $\ell_2^{-1} \circ \ell_1$.  Then, for algebras $(A_1, O_1, \ell_1)$ and $(A_2, O_2, \ell_2)$ of the same type, we call a map $g:A_1 \to A_2$ an \emph{algebra homomorphism} if, for every $f \in O_1$, we have that, $g(f(a_1,\ldots, a_{N_f})) = \tilde{f}(g(a_1), \ldots, g(a_{N_f}))$ for all $a_1, \ldots, a_{N_f} \in A_1$, where $\tilde{f}:= \ell_2^{-1} \circ \ell_1(f)$.

\subsubsection{Truth-value Algebras}

We define a \emph{truth-value algebra} of type $(\alpha,\tau)$ to be a tuple $(A,\leq,O,\ell)$ where $(A,O,\ell)$ is an algebra of type $(\alpha,\tau)$, and where $(A,\leq)$ is a poset with a top element.  When $\alpha$ is a finite set with largest element $n$, we simply say that the truth-value algebra is of type $(\tau(0), \tau(1), \ldots, \tau(n))$.  Also, when it will not lead to confusion, we will refer to the set $A$ as the truth-value algebra.

Note that there is a trivial truth-value algebra (for any given type $(\alpha,\tau)$) consisting of the single element $1$, where every $f \in O$ is the trivial map, and we denote this truth-value algebra by $\mathbf{1}_{(\alpha,\tau)}$, or just by $\mathbf{1}$ when the type is clear from the context.  

Boolean algebras are prototypical examples of truth-value algebras, having the type $(2,2,1)$.  Explicitly, a Boolean algebra is a tuple $(B,\leq,O,\ell)$, where $O = \{ \wedge, \vee, \neg \}$, and where $\ell(\wedge) = 0$, $\ell(\vee) = 1$ and $\ell(\neg) = 2$ (so that $N_{\wedge} = N_{\vee} = 2$ and $N_{\neg} = 1$). Typically, for $a,b \in B$, we write $a \wedge b$ rather than $\wedge(a,b)$, and $a \vee b$ rather than $\vee(a,b)$.  The following constraints must be satisfied for all $a,b,c \in B$.

\begin{enumerate}[(i)]
\item $\bigwedge \{a,b\}$ exists and $a \wedge b = \bigwedge \{a,b\}$,
\item $a \leq b$ implies that $\neg b \leq \neg a$,
\item $\neg (\neg a) = a$,
\item $a \vee \neg a = 1$
\item $a \vee (b \wedge c) = (a \vee b) \wedge (a \vee c)$.
\end{enumerate}

The partial order then satisfies $a \leq b$ whenever $a \wedge b = a$ for all $a,b \in B$.  In addition to a top element, $B$ also has a bottom element $0 := \neg 1$.  $\mathbf{1}$ is easily seen to be a Boolean algebra.  The smallest non-trivial Boolean algebra consists of $B = \{0,1\}$ with the standard operations, and we denote this Boolean algebra by $\mathbf{2}$.

We will need to define a few simple notions concerning truth-value algebras.  First, for truth-value algebras $(A_1, \leq_1, O_1, \ell_1)$ and $(A_2, \leq_2, O_2, \ell_2)$ of the same type, a \emph{homomorphism} is a map $g:A_1 \to A_2$ which is both an algebra homomorphism (from $(A_1, O_1, \ell_1)$ to $(A_2, O_2, \ell_2)$), as well as an isotone, continuous, and top-preserving map (from $(A_1, \leq_1)$ to $(A_2, \leq_2)$).  If a homomorphism is 1-1 and onto, it is called an isomorphism, and  whenever there exists some isomorphism $g:A_1 \to A_2$,  then $A_1$ and $A_2$ are said to be \emph{isomorphic} (denoted $A_1 \simeq A_2$).  

We also note that one can define products of truth-value algebras of the same type in the usual way.  Given two truth-value algebras $(A_1, \leq_1, O_1, \ell_1)$ and $(A_2, \leq_2, O_2, \ell_2)$, we can give the (set theoretic) Cartesian product $A_1 \times A_2$ the structure of a truth-value algebra $(A_1 \times A_2, \leq, O, \ell)$ by defining $(a_1, a_2) \leq (b_1, b_2)$ exactly when $ a_1 \leq_1 b_1$ and $a_2 \leq_2 b_2$, and for each $f \in O_1$, (letting $\tilde{f} := \ell_2^{-1} \circ \ell_1(f)$) defining $\hat{f}((a_1,b_1),  (a_2,b_2), \ldots, (a_{N_f}, b_{N_f})) := (f(a_1, \ldots, a_{N_f}),\tilde{f}(b_1, \ldots, b_{N_f})) $.  Then we define $O$ to be the collection of all such $\hat{f}$'s (for $f \in O_1$), and also set $\ell(\hat{f}) := \ell_1(f)$.  

Then, for any truth-value algebras $A_1$ and $A_2$ of the same type, we have $A_1 \times A_2 \simeq A_2 \times A_1$, as well as that $A \times \mathbf{1} \simeq A$ for any truth-value algebra $A$.  Also, for $A_1 \times A_2$, we have the natural projection maps $p_i:A_1 \times A_2 \to A_i$ for $i =1,2$, which behave as one expects.

\begin{prop}
\label{prop:p_hom}
Let $A_1$ and $A_2$ be truth-value algebras of the same type.  Then the projection maps $p_i:A_1 \times A_2 \to A_i$ for $i =1,2$ are both homomorphisms.
\end{prop}
\begin{proof}
The only non-trivial property to prove is the continuity of $p_i$, so consider $S \subseteq A_1 \times A_2$ such that $\bigwedge S$ exists, and define $(a,b) := \bigwedge S$.  We consider $p_1$, the proof for $p_2$ is similar.  

Now, for all $(c, d) \in S$, we have that $(a,b) \leq (c,d)$, and hence $a \leq c$, so that $a$ is a lower bound for the set of all $c$ such that $(c,d) \in S$ for some $d \in A_2$ (i.e.~for $p_1(S)$).  Also, for any $c' \in A_1$ which is a lower bound for $p_1(S)$, we have that $(c ', b) \leq (c, d)$ for all $(c, d) \in S$, and hence $(c ' ,b) \leq (a,b)$, so that $c ' \leq a$.  Hence $a$ is the greatest lower bound for $p_1(S)$, i.e.~$p_1(\bigwedge S) = \bigwedge p_1 (S)$.
\end{proof}

The following concept will be central in what follows.

\begin{kap}
Let $A$ be a truth-value algebra.  Then $A$ is said to be \emph{irreducible} if whenever $A~\simeq~A_1~\times~A_2$, we have that either $A_1 \simeq \mathbf{1}$ or $A_2 \simeq \mathbf{1}$.
\end{kap}

\subsection{First-Order Algebraic Logic}
\label{FOalgebraiclogic}

\subsubsection{Syntax}

The syntax we use for first-order algebraic logic is similar to that typically used for first-order classical logic, the only difference is that the propositional logical connectives may be different.

First, fix a type $T = (\alpha, \tau)$.  We define a \emph{$T$-language} $\mathcal{L}$ to consist of a set of predicate, function, and propositional connective symbols, which we assume are disjoint and denote by $\mathcal{L_P}$, $\mathcal{L_F}$, and $\mathcal{L_C}$, respectively, so that $\mathcal{L} = \mathcal{L_P} \cup \mathcal{L_F} \cup \mathcal{L_C}$, and where $\mathcal{L_C}$ is in 1-1 correspondence with $\alpha$.   Each element of a $T$-language has an associated non-negative integer, called its $\emph{arity}$ (functions of arity $0$ are said to be \emph{constants}), and the arity of a given logical connective $c_i \in \mathcal{L_C}$ (corresponding to $i \in \alpha$) is defined to be $\tau(i)$.

We define our variables to run over some countable set, and use the symbols $x,y,z,x_1, x_2, \ldots$ as meta-variables, i.e.~variables which each stand for some variable in our aforementioned countable set.  We then define $\mathcal{L}$-\emph{terms} in the standard way starting with variables and constants as atomic terms and proceeding inductively using function symbols of higher arity.  Similarly,  we define well-formed formulas for a given language $\mathcal{L}$ (which we call $\mathcal{L}$-\emph{wffs}) inductively using the propositional connectives along with the universal quantifier ``$\forall$'' in the standard manner.\footnote{We restrict our attention to first-order logics where the existential quantifier ``$\exists$'' only occurs as a defined symbol; for example, in standard first-order logic with Boolean connectives, $\exists := \neg \forall \neg \,$.  One could extend the results of this paper to first-order logics with ``$\exists$'' defined independently by using analogous constraints involving least upper bounds wherever we have used greatest lower bounds in order to deal with  ``$\forall$''.}

We also use the usual notion of free and bound variables, and for an $\mathcal{L}$-wff $\psi$, we write $\psi(x_1, \ldots, x_n)$ to mean that the free variables occurring in $\psi$ are among $x_1, \ldots, x_n$ (although not all of them may actually occur free in $\psi$).  For an $n$-tuple of constants $(a_1, \ldots, a_n )$, we then write $\psi(a_1, \ldots, a_n)$ for that $\mathcal{L}$-wff which is obtained from $\psi$ by substituting $a_i$ for each free occurrence of $x_i$, with $i \in \{1, \ldots, n\}$. If an $\mathcal{L}$-wff has no free variables, we say that it is an \emph{$\mathcal{L}$-sentence}.

\subsubsection{Models for Algebraic Logic}
\label{modelsAL}

For any $T$-language $\mathcal{L}$, we define an $\mathcal{L}$-structure for algebraic logic in a fashion analogous to that for classical logic, except that we allow any truth-value algebra of type $T$ to serve as the allowed truth values, rather than any Boolean algebra (or just $\mathbf{2}$, depending on the particular development of classical logic).  

In order to accomplish this definition, we first need a language rich enough to refer to elements of a given structure.  To this end, given a $T$-language $\mathcal{L}$, and a set $M$ (disjoint from $\mathcal{L}$), we define $\mathcal{L}^M$ to be the $T$-language given by $\mathcal{L} \cup M$, where each element of $M$ is taken to be a function symbol of arity $0$ (i.e.~a constant).  Note that since $\mathcal{L} \subseteq \mathcal{L}^M$, every $\mathcal{L}$-wff is an $\mathcal{L}^M$-wff.  Next, we need to interpret our function symbols as actual maps in our models, which motivates the following notion.  For a given language $\mathcal{L}$, associated function symbols $\mathcal{L_F}$, and a given set $M$, an \emph{interpretation of $\mathcal{L}$ in $M$} is a map which assigns to each function symbol $f \in \mathcal{L_F}$ of arity $n$ a map from $M^n \to M$.  Note that if $F$ is an interpretation of $\mathcal{L}$ in $M$, we can always extend $F$ to an interpretation of $\mathcal{L}^M$ in $M$ by taking $F(a) = a$ for each $a \in M$.  Whenever necessary in the sequel, we will consider $F$ to be so extended without further comment.  Then, given any $\mathcal{L}^M$-term $t$ containing no variables, we can map this term to $M$ in the usual way, which we denote by $F(t)$.

Finally, we need to assign truth values (elements of a given truth-value algebra $A$) to our $\mathcal{L}^M$-wffs.  In order to do this in a way which respects the usual meaning of our logical connectives we make the following definition.

\begin{kap}
\label{def:truth_valuation}
Let $\mathcal{L}$ be a $T$-language, $A$ a truth-value algebra with labeling function $\ell$, $M$ a set, and $F$ an interpretation of $\mathcal{L}$ in $M$.  Then a \emph{truth valuation for $F$ in $A$} is a surjective\footnote{We could have equally well proceeded to develop our model theory without the condition of surjectivity --- however this assumption will streamline some of the presentation in the sequel.} map $\nu$ from the set of $\mathcal{L}^M$-sentences to $A$ which satisfies the following properties:
\begin{enumerate}[(1)]
\item $\nu \big( P(t_1, \ldots, t_n) \big) = \nu \big( P ( F(t_1), \ldots, F(t_n) ) \big) $ for any predicate $P \in \mathcal{L_P}$ of arity $n$ and $\mathcal{L}^M$-terms $t_1, \ldots, t_n$ with no variables;
\item $\nu (c_i(\psi_1, \ldots, \psi_n) ) = f_i ( \nu( \psi_1), \ldots , \nu(\psi_n) )$ for any connective $c_i \in \mathcal{L_C}$ of arity $n$, where $f_i = \ell^{-1}(i)$, and for any $\mathcal{L}^M$-sentences $\psi_1, \ldots, \psi_n$;
\item Whenever $\displaystyle \bigwedge_{a \in M} \nu \big( \psi(a) \big)$ exists for a given $\mathcal{L}^M$-wff $\psi(x)$, we have that \\ $\displaystyle \nu \big( (\forall x) \psi (x) \big) =  \bigwedge_{a \in M} \nu \big( \psi(a) \big)$;
\item $\displaystyle \bigwedge_{a \in M} \nu \big( \psi(a) \big) $ exists for any $\mathcal{L}^M$-wff $\psi(x)$.
\end{enumerate}
(The break-up of statements (3) and (4) above may seem odd --- we do this for ease of reference in the sequel.)  We can then define $\nu(\psi(x_1, \ldots, x_n)) = \nu \big( (\forall x_1) \cdots (\forall x_n) \psi(x_1, \ldots, x_n) \big)$ for any $\mathcal{L}^M$-wff $\psi$, thus, extending $\nu$ to all $\mathcal{L}^M$-wffs.  We will always consider our truth valuations to be so extended.
\end{kap}

For any $T$-language $\mathcal{L}$, we now define an \emph{$\mathcal{L}$-structure}  to be a tuple $(M,A,\nu, F)$ where $M$ is a set (called the \emph{universe}), $A$ is a truth-value algebra of type $T$, $F$ is an interpretation of $\mathcal{L}$ in $M$, and $\nu$ is a truth valuation for $F$ in $A$.  Then, given a set of $\mathcal{L}$-wffs $\Gamma$, and an $\mathcal{L}$-structure $\mathfrak{M} := (M,A, \nu, F)$, we say that $\mathfrak{M}$ is a \emph{model for $\Gamma$} if $\nu \big( \gamma(a_1, \ldots, a_n) \big) = 1$ for every $\gamma(x_1, \ldots, x_n) \in \Gamma$ and $a_1, \ldots, a_n \in M$.  For any $\mathcal{L}$-structure $\mathfrak{M}$ with truth valuation $\nu$, we will say that an $\mathcal{L}$-wff $\gamma$ \emph{holds in $\mathfrak{M}$} (or write $\mathfrak{M} \vDash \gamma$) if $\nu(\gamma) = 1$.  This means that $\mathfrak{M}$ is a model of $\Gamma$ if and only if every $\gamma \in \Gamma$ holds in $\mathfrak{M}$.  In what follows, we will be particularly interested in models for which $A$ is an irreducible truth-value algebra, motivating the following definition.

\begin{kap}
\index{irreducible!model}
\index{model!irreducible}
\label{def:irreducible}
Let $\mathcal{L}$ be a $T$-language, and let $\mathfrak{M}$ be an $\mathcal{L}$-structure.  If the truth-value algebra of $\mathfrak{M}$ is irreducible, then $\mathfrak{M}$ is called an \emph{irreducible $\mathcal{L}$-structure} (or \emph{irreducible model}, if $\mathfrak{M}$ is a model).
\end{kap}

The truth of a given $\mathcal{L}$-wff will be decided with respect to an allowed class of models, which motivates the formal notion of a semantics for a logic.

\begin{kap}
Let $\mathcal{L}$ be a $T$-language, and let $\mathcal{T}$ be a class of $\mathcal{L}$-structures.  Then $\mathcal{T}$ is said to be an \emph{$\mathcal{L}$-semantics}.  
\end{kap}

Whenever a given model (or structure) has a product truth-value algebra, we can use the natural projection maps on the truth-value algebras to construct new models (or structures).

\begin{prop}
\label{prop:product_model}
Let $\mathcal{L}$ be a $T$-language, let $\Gamma$ be a set of $\mathcal{L}$-wffs, and let $\mathfrak{M} = (M, A_1 \times A_2, \nu , F)$ (where $A_1,A_2$ are both truth-value algebras of type $T$) be a model of $\Gamma$.  Then $\mathfrak{M}_i := (M, A_i, p_i \circ \nu, F)$ is a model of $\Gamma$, where $p_i:A_1 \times A_2 \to A_i$ is the natural projection map onto $A_i$ (with $i = 1,2$).  
\end{prop}
\begin{proof}
For any $\gamma \in \Gamma$, $\nu(\gamma) = 1$ (since $\mathfrak{M}$ is a model of $\Gamma$), and so $p_i \circ \nu(\gamma) = 1$ (the top element of $A_1 \times A_2$ is easily seen to be $(1,1)$).  Also, since $\nu$ and $p_i$ are both surjective,  $p_i \circ \nu$ is also surjective.  Moreover, conditions (1)-(3) of def.~\ref{def:truth_valuation} are satisfied since $\nu$ is a truth valuation and $p_i$ is a homomorphism.  It remains to show only that $p_i \circ \nu$ satisfies condition (4) of that same definition.

To this end, consider an $\mathcal{L}$-wff $\psi(x)$, and let $(\gamma, \beta) = \bigwedge_{a \in M}\nu(\psi(a))$.  Since $p_i$ is a homomorphism (proposition \ref{prop:p_hom}), it is, in particular, continuous, and hence $\gamma$ is the greatest lower bound for the set $\{ p_i \circ \nu(\psi(a)) \ : \ a \in M \}$, establishing condition (4) of def.~\ref{def:truth_valuation}.
\end{proof}

With this in mind, one natural property that an $\mathcal{L}$-semantics can possess is the following:

\begin{kap}
Let $\mathcal{L}$ be a $T$-language, and let $\mathcal{T}$ be an $\mathcal{L}$-semantics.  Then $\mathcal{T}$ is said to be \emph{factor-closed} when for every $\mathcal{L}$-structure $(M,A_1 \times A_2, \nu, F)   \in \mathcal{T}$, we have that both $(M,A_1,p_1 \circ \nu, F)$ and $(M,A_2,p_2 \circ \nu, F)$ are in $\mathcal{T}$, where $p_i:A_1 \times A_2 \to A_i$ (for $i = 1,2$) are the natural projection maps.
\end{kap}

\subsubsection{Formal Deduction for Algebraic Logic}

For a given $T$-language $\mathcal{L}$, an \emph{$\mathcal{L}$-rule} is a pair $(\Gamma, \gamma)$, where $\Gamma$ is a finite set of $\mathcal{L}$-wffs, and $\gamma$ is a single $\mathcal{L}$-wff.  This is a formalization of the idea that ``$\gamma$ follows from $\Gamma$''.  A \emph{formal deduction system for $\mathcal{L}$} is then a pair $(\mathcal{A},\mathcal{R})$, where $\mathcal{A}$ is a set of $\mathcal{L}$-wffs (representing ``logical axioms''), and $\mathcal{L}$ is a set of $\mathcal{L}$-rules (the ``rules of inference'').  

A \emph{formal proof in a given formal deduction system $(\mathcal{A}, \mathcal{R})$ from $\Gamma$} (where $\Gamma$ is a set of $\mathcal{L}$-wffs, representing the hypotheses of the proof) is then a finite tuple $(\gamma_1, \ldots, \gamma_n)$ of $\mathcal{L}$-wffs such that, for each $\gamma_i$, $i=1, \ldots, n$, we have either (i) $\gamma_i \in \mathcal{A} \cup \Gamma$, or (ii) there exists a set $\Gamma_0 \subseteq \{ \gamma_1, \ldots, \gamma_{i-1} \}$ for which $(\Gamma_0, \gamma_i) \in \mathcal{R}$.  If a formal proof for $\gamma$ in a deduction system $\mathcal{D}$ exists from some set of $\mathcal{L}$-wffs $\Gamma$, we write $\Gamma \vdash \gamma$ (where we eschew a $\mathcal{D}$-dependent notation since the deduction system will always be clear from the context).  We now formally define the notions of soundness and completeness for our logics.

\begin{kap}
\label{thm:completeness}
Let $\mathcal{L}$ be a $T$-language, let $\mathcal{T}$ be an $\mathcal{L}$-semantics, and let $\mathcal{D}$ be a formal deduction system for $\mathcal{L}$.
\begin{enumerate}
  \item $\mathcal{D}$ is said to be \emph{complete for $\mathcal{T}$} if, for any set of $\mathcal{L}$-wffs $\Gamma$, whenever $\psi$ is an $\mathcal{L}$-wff satisfying $\mathfrak{M} \vDash \psi$ for every $\mathfrak{M} \in \mathcal{T}$ which is an $\mathcal{L}$-model of $\Gamma$, we have $\Gamma \vdash \psi$.
  \item $\mathcal{D}$ is said to be \emph{sound for $\mathcal{T}$} if, for any set of $\mathcal{L}$-wffs $\Gamma$, whenever $\psi$ is an $\mathcal{L}$-wff satisfying $\Gamma \vdash \psi$, then for any model $\mathfrak{M} \in \mathcal{T}$ of $\Gamma$, we have that $\mathfrak{M} \vDash \psi$.
\end{enumerate}
\end{kap}

\section{Main Theorem}

\label{sec:main}

We will now prove our main result.

\vspace{.2in}

\begin{mainthrm*}
\label{thm:irreducibles_suffice}
Let $\mathcal{L}$ be a $T$-language, let $\mathcal{T}$ be an $\mathcal{L}$-semantics which is factor-closed, and let there exist some formal deduction system $\mathcal{D}$ which is both sound and complete for $\mathcal{T}$.  Further let $\Gamma$ be a set of $\mathcal{L}$-wffs, and let $\psi$ be an $\mathcal{L}$-wff which holds in every irreducible model of $\Gamma$ which is in $\mathcal{T}$.  Then $\psi$ holds in every model of $\Gamma$ which is in $\mathcal{T}$.
\end{mainthrm*}
\begin{proof}
We prove the contrapositive, so let $\psi$ be an $\mathcal{L}$-wff and assume that there is some model of $\Gamma$ in $\mathcal{T}$ in which $\psi$ does not hold.  We will show that there is some irreducible model of $\Gamma$ in $\mathcal{T}$ in which $\psi$ does not hold.

By the assumption of completeness, we must have that $\Gamma$ does not prove $\psi$ in $\mathcal{D}$.  Now define $\mathcal{S}_\Gamma$ to be the set which consists of exactly those sets of $\mathcal{L}$-wffs $\Phi$ which satisfy both (i) $\Gamma \subseteq \Phi$, and (ii)~$\nu(\psi) \neq 1$ for some truth valuation $\nu$ in some model of $\Phi$ which is in $\mathcal{T}$.  Clearly $\mathcal{S}_\Gamma$ is non-empty.

$\mathcal{S}_\Gamma$, like any set of sets, is partially ordered by inclusion, and we will use Zorn's lemma to prove that $\mathcal{S}_\Gamma$ contains some maximal element $\Omega$.  Let $\mathcal{C} \subseteq \mathcal{S}_\Gamma$ be linearly ordered under inclusion, and define $U := \bigcup_{c \in \mathcal{C}} c $.  Clearly $U$ is an upper bound for $\mathcal{C}$ under inclusion, so it suffices to show that $U \in \mathcal{S}_\Gamma$.  Property (i) which defines $\mathcal{S}_\Gamma$ clearly holds, since trivially $\Gamma \subseteq U$.  We prove property (ii) by contradiction, so assume $\nu(\psi) = 1$ for every truth valuation $\nu$ in every model of $U$, and hence $U \vdash \psi$ by the completeness of $\mathcal{D}$.  The formal proof of $U \vdash \psi$ is finite by definition, and so may only contain a finite number of $\mathcal{L}$-wffs $\gamma_1, \gamma_2, \ldots, \gamma_n \in U$, where each $\gamma_i \in \Phi_i$ for some $\Phi_i \in \mathcal{C}$.  Since $\mathcal{C}$ is linearly ordered by inclusion, we must have that $\bigcup_{i=1}^n \Phi_i = \Phi_j$ for some $j \in \{1, \ldots, n\}$, and since $\{\gamma_1, \ldots, \gamma_n \} \subseteq \Phi_j$, we have that $\Phi_j \vdash \psi$, so that, by the soundness of $\mathcal{D}$, $\nu(\psi) = 1$ for every truth valuation $\nu$ in every model of $\Phi_j$, which is a contradiction since $\Phi_j \in \mathcal{S}_{\Gamma}$.  Hence we must have that $U \in \mathcal{S}_\Gamma$, and since $\mathcal{C}$ was an arbitrary chain in $\mathcal{S}_\Gamma$, Zorn's lemma gives that $\mathcal{S}_\Gamma$ has a maximal element $\Omega$.

Since $\Omega \in S_{\Gamma}$, there exists a model $\mathfrak{M} := (M, A, \nu, F)$ of $\Omega$ in $\mathcal{T}$ in which $\nu(\psi) \neq 1$.  We claim that the truth-value algebra $A$ in this model is irreducible.  Assume not, so that $A \simeq A_1 \times A_2$ with both $A_1, A_2$ non-trivial.  By proposition \ref{prop:product_model}, this produces two new models of $\Omega$ with corresponding truth-value algebras $A_1$ and $A_2$ (and truth valuations $\nu_1 := p_1 \circ \nu$ and $\nu_2 := p_2 \circ \nu$, respectively).  Of course, these models are both in $\mathcal{T}$, since $\mathcal{T}$ is factor-closed.  Since $\nu(\psi) \neq 1$, we must have (without loss of generality) $\nu_1(\psi) \neq 1$.  Let $\chi$ be some $\mathcal{L}$-wff such that $\nu_1(\chi) = 1$ and $\nu_2(\chi) \neq 1$.  Clearly, $\chi \notin \Omega$, since $\nu(\chi) \neq 1$.  Now, we cannot have $\Omega \cup \{ \chi \} \vdash \psi$, since then we would have $\nu_1(\psi) = 1$.   Hence, we must have that $\Omega \cup \{ \chi \} \in \mathcal{S}_\Gamma$, but this means that $\Omega$ is a proper subset of $ \Omega \cup \{ \chi \}$, which contradicts the maximality of $\Omega$ in $\mathcal{S}_\Gamma$.  Hence, we must have that $A$ is irreducible.  This shows that $\mathfrak{M}$ is an irreducible model of $\Gamma$ in which $\nu(\psi) \neq 1$.
\end{proof}

So we see that for any first-order algebraic logic (as defined in section \ref{FOalgebraiclogic}) which is formulated with respect to a factor-closed class of algebras, a Hilbert-style deduction system is sound and complete for that class of algebras if and only if it is sound and complete for the irreducible algebras in that class.  In particular, this result applies to quantum logic, to which we now turn our attention.

\pagebreak

\section{First-Order Quantum Logics}
\label{QL}

Any first-order quantum logic (at least as considered in this paper) is defined through its model theory, and the model theory is entirely analogous to that of classical logic, except that any fixed variety of orthomodular lattices is allowed to serve as the class of truth-value algebras as opposed to only Boolean algebras \cite{Dishkant_74, Greechie}.
In section \ref{OMLs} we briefly discuss orthomodular lattices (see~\cite{Kalmbach}), and in section \ref{detailsQL} we precisely define first-order quantum logics and apply our main theorem.

\subsection{Orthomodular Lattices}
\label{OMLs}

An \emph{orthomodular lattice} is a tuple $(L,\leq,O,\ell)$, where $(L,\leq)$ is a poset with a top element, $O = \{ \wedge, \vee, \neg \}$, and $\ell(\wedge) = 0$, $\ell(\vee) = 1$ and $\ell(\neg) = 2$ (so that $N_{\wedge} = N_{\vee} = 2$ and $N_{\neg} = 1$,
i.e.~orthomodular lattices have type $(2,2,1)$).  Additionally, the following constraints\footnote{If only (ii)-(iii) are satisfied, $L$ is called an \emph{involutive lattice}, while if (ii)-(iv) are satisfied, it is called an \emph{ortholattice}.} must be satisfied for all $a,b \in L$.
\begin{enumerate}[(i)]
\item $\bigwedge \{a,b\}$ exists and $a \wedge b = \bigwedge \{a,b\}$,
\item $a \leq b$ implies that $\neg b \leq \neg a$,
\item $\neg (\neg a) = a$,
\item $a \vee \neg a = 1$
\item $a \leq b$ implies $a \vee (\neg a \wedge b) = b$.
\end{enumerate}
Boolean algebras are special cases of orthomodular lattices --- to wit, Boolean algebras are exactly the orthomodular lattices $(L,\leq,O,\ell)$ in which all pairs of elements are \emph{compatible}, i.e.~for all $a,b \in L$,
\begin{enumerate}[(vi)]
\item $a = (a \wedge b) \vee (a \wedge \neg b)$.
\end{enumerate}
In particular, both $\mathbf{1}$ and $\mathbf{2}$ are orthomodular lattices.

In an orthomodular lattice $L$, the \emph{center of $L$} (denoted $\mathcal{Z}(L)$) is the subset of $L$ consisting of elements of $L$ which are compatible with every element in $L$.
Note that for any orthomodular lattice $L$, $\mathcal{Z}(L)$ is a Boolean sub-algebra of $L$.
A characterization of irreducibility in orthomodular lattices is given by the following proposition.
\begin{prop}
Let $L$ be an orthomodular lattice. Then $L$ is \emph{irreducible} if and only if $\mathcal{Z}(L) = \{ 0, 1 \}$.
\end{prop}
This is to say that the irreducible orthomodular lattices are exactly those orthomodular lattices which have, as their center, the only irreducible Boolean algebra $\mathbf{2}$.

\subsection{Irreducibles Suffice for Quantum Logic}
\label{detailsQL}

For a given quantum logic, we fix type $(2,2,1)$, where $\mathcal{L_C} = \{ \wedge, \vee, \neg \}$,\footnote{We use the same symbols for the logical connectives as for operations in orthomodular lattices, but no confusion should result.} and we let $\mathcal{V}$ denote the variety of orthomodular lattices with which the quantum logic is associated.  In this context, the relevant $\mathcal{L}$-structures are those tuples $(M,L,\nu, F)$ (as defined in section \ref{modelsAL}), where the truth-value algebra $L$ is an orthomodular lattice; models are defined analogously --- for the full details see \cite{DeJonghe_13, Frey_13}.
Per definition \ref{def:irreducible}, an $\mathcal{L}$-structure is irreducible exactly when the truth-value algebra $L$ is an irreducible orthomodular lattice.
We obtain an $\mathcal{L}$-semantics $\mathcal{T}_{\mathcal{V}}$ for the quantum logic by taking $\mathcal{T}_{\mathcal{V}}$ to be the class of $\mathcal{L}$-structures such that the truth-value algebras run over all orthomodular lattices in $\mathcal{V}$.  (If $\mathcal{V}$ is just the variety of Boolean algebras, this yields classical logic.)
Note that any  variety is closed under homomorphic images (by Birkhoff's HSP theorem \cite{Burris}); as such, we have that $\mathcal{T}_{\mathcal{V}}$ is factor closed.  In 1974 Dishkant constructed a formal deduction system for a quantum logic based on the  variety $\mathcal{V}$ consisting of all orthomodular lattices, and he demonstrated that it was both sound and complete for $\mathcal{T}_{\mathcal{V}}$ \cite{Dishkant_74}; an alternative system was provided in \cite{IDF, DeJonghe_13, Frey_13}. An analogous result can easily be seen to hold for any variety of orthomodular lattices.

Now, since a given quantum logic satisfies the assumptions of the Main Theorem, we see that a formal deduction system is both sound and complete for $\mathcal{T}_{\mathcal{V}}$ if and only if it is sound and complete for a semantics $\mathcal{T}^{0}_{\mathcal{V}}$ utilizing only irreducible orthomodular lattices in $\mathcal{V}$.  Note, in particular, that each truth-value algebra $L_0$ associated with a model in $\mathcal{T}^{0}_{\mathcal{V}}$ has $\mathcal{Z}(L_0)= \mathbf{2}$, so that taking $\mathcal{V}$ to be the variety of orthomodular lattices consisting of Boolean algebras, we recover exactly the original theorem (stated in section \ref{intro}) for classical logic.

Clearly, our Main Theorem greatly simplifies investigations into the model theory of any quantum logic, as it allows one to focus solely on irreducible orthomodular lattices in $\mathcal{V}$, rather than \emph{all} of~$\mathcal{V}$. We note that the projection lattice of a (separable) Hilbert space of any given dimension (even infinite) is an irreducible orthomodular lattice, and the primary reason for interest in logics based on projection lattices stems from their relevance to quantum theory itself \cite{Greechie}. For some concrete applications of our result to the development of mathematics based on quantum logic --- in particular set theory and arithmetic --- see \cite{IDF, DeJonghe_13, Frey_13}.

\section{Open Questions}

\label{sec:conc}

Just as we have proved a more powerful completeness result allowing the restriction of our semantics to irreducible models, we would like to know if there are other sub-classes of models for which an analogous completeness result holds.  In the classical case, we have completeness with respect to the single Boolean algebra $\mathbf{2}$, and this algebra satisfies other properties besides irreducibility, such as being subdirectly irreducible \cite{Burris} and being a complete lattice with respect to the partial order.  We would like to know which of our first-order algebraic logics are complete with respect to models whose truth-value algebras are either all subdirectly irreducible or all complete lattices, just as they are with respect to irreducible models.

\bibliographystyle{unsrt}
\bibliography{bibliography}

\end{document}